\documentclass{amsart}
\usepackage{amsmath,amssymb,amscd,latexsym}

\usepackage[all]{xy}
\newcommand{\A}{{\mathbb{A}}}

\newcommand{\Q}{{\mathbb{Q}}}

\newcommand{\Z}{{\mathbb{Z}}}

\newcommand{\et}{\mathrm{\acute{e}t}}
\newcommand{\ret}{\mathrm{r\acute{e}t}}

\newcommand{\bpi}{\bar{\pi}}
\newcommand{\ok}{\bar{k}}

\newcommand{\Xok}{X_{\ok}}

\newcommand{\bi}{\bar{i}}
\newcommand{\bI}{\bar{I}}
\newcommand{\pro}{{\mathrm{pro}}}

\newcommand{\Spec}{\mathrm{Spec}\,}
\newcommand{\colim}{\operatorname*{colim}}

\newcommand{\holim}{\operatorname*{holim}}
\newcommand{\cosk}{\mathrm{cosk}}

\newcommand{\Gal}{\mathrm{Gal}}

\newcommand{\Hom}{\mathrm{Hom}}

\newcommand{\Map}{\mathrm{Map}}

\newcommand{\Ch}{{\mathcal C}}

\newcommand{\Gh}{{\mathcal G}}

\newcommand{\hHhp}{\hat{{\mathcal H}}_{\ast}}

\newcommand{\Sh}{{\mathcal S}}
\newcommand{\Shp}{{\mathcal S}_{\ast}}

\newcommand{\hSh}{\hat{\mathcal S}}

\newcommand{\hShpg}{\hat{\mathcal S}_{\ast G}}

\newcommand{\hShp}{\hat{\mathcal S}_{\ast}}

\newcommand{\Xh}{\mathcal{X}}

\newcommand{\bXh}{\bar{\Xh}}
\newcommand{\Xhpf}{\Xh_{\mathrm{pf}}}
\newcommand{\Yh}{\mathcal{Y}}

\newtheorem{theorem}{Theorem}[section]
\newtheorem{lemma}[theorem]{Lemma}
\newtheorem{prop}[theorem]{Proposition}
\newtheorem{defn}[theorem]{Definition}

\newtheorem{conjecture}[theorem]{Conjecture}
\theoremstyle{definition}

\newtheorem{example}[theorem]{Example}

\newtheorem{remark}[theorem]{Remark}
\begin{document}

\title{Existence of rational points as a homotopy limit problem}

\author{Gereon Quick}

\thanks{The author was supported in part by the German Research Foundation (DFG)-Fellowship QU 317/1}

\address{Mathematisches Institut, WWU M\"unster, Einsteinstr. 62, 48149 M\"unster, Germany}

\email{gquick@math.uni-muenster.de}

\date{}

\begin{abstract}
We show that the existence of rational points on smooth varieties over a field can be detected using homotopy fixed points of \'etale topological types under the Galois action. As our main example we show that the surjectivity statement in Grothendieck's Section Conjecture would follow from the surjectivity of the map from fixed points to continuous homotopy fixed points on the level of connected components. Along the way we define a new model for the continuous \'etale homotopy fixed point space of a smooth variety over a field under the Galois action. 
\end{abstract}

\maketitle

\section{Introduction}

%{\bf WARNING: $G$ does NOT act on $\pi_1(\Xok,\bar{x})$ unless we already assume that $\bar{x}$ is a rational point, i.e. fixed under $G$! So we CANNOT just take $K(\pi_1(X,\bar{x}),1)$ as a model in $\hShg$ for $\hEt \Xok$! }\\

Let $k$ be a field and $X$ a variety over $k$. To find all $k$-rational points of $X$ is an important  and often very difficult problem. Many techniques have been developed to either prove the existence or non-existence of rational points. Recently, several topological approaches have been established for example in \cite{ah}, \cite{pal1}, \cite{pal2}, \cite{kirsten2}. In particular, Harpaz-Schlank showed in \cite{hs} that certain obstructions to the existence of rational points can be formulated in terms of homotopy fixed points under the Galois action. 

In this paper we continue the independent approach in \cite{gspaces} and show that also the existence of rational points can be detected via continuous homotopy fixed points under the Galois action. As the main example and motivation for this approach we briefly recall Grothendieck's section conjecture which is one of the most important open problems on rational points. 

Let $\ok$ be an algebraic closure of $k$, $G:=\Gal(\ok/k)$ and $X$ be a geometrically connected variety over $k$ equipped with a geometric point $x$. Let $\Xok$ be the lift of $X$ to $\ok$. Taking \'etale fundamental groups $\pi_1(-,x)=\pi_1^{\et}(-,x)$ induces a short exact sequence of profinite groups 
\begin{equation}\label{sesintro}
1 \to \pi_1(\Xok, x) \to \pi_1(X, x) \to G \to 1.
\end{equation}
If $a\in X(k)$ is a $k$-rational point on $X$, then the functoriality of $\pi_1$ induces a continuous section $\sigma_a:G_k \to \pi_1(X, x)$ of \eqref{sesintro} which is well-defined up to conjugation by elements in $\pi_1(\Xok, x)$.  Grothendieck's Section Conjecture predicts that this map has an inverse in the following case (see also \cite{stixbook} for more details in this conjecture).

\begin{conjecture}\label{sconj} {\rm (Grothendieck \cite{grothendieck})}
Let $k$ be a field which is finitely generated over $\Q$ and let $X$ be a smooth, projective curve of genus at least two. The map $a \mapsto [\sigma_a]$ is a bijection between the set $X(k)$ of $k$-rational points of $X$ and the set $S(\pi_1(X/k))$ of $\pi_1(\Xok, x)$-conjugacy classes of continuous sections $G_k \to \pi_1(X, x)$.
\end{conjecture}

It is well-known that the map $a \mapsto [\sigma_a]$ is injective. Hence the conjecture is a statement about the existence of rational points. The main result of this paper is that the surjectivity of the map $a \mapsto [\sigma_a]$ would follow from the solution of a homotopy limit problem in the spirit of the Sullivan Conjecture. \\

We now outline the main ideas of the paper. 
Let $k$ be an arbitrary field with algebraic closure $\ok$, $G:=\Gal(\ok/k)$ and $X$ be a quasi-projective geometrically connected smooth variety over $k$. %Taking \'etale fundamental groups factors through forming \'etale homotopy types.
A $k$-rational point $a: \Spec k \to X$ induces a map of \'etale homotopy types $(\Spec k)_{\et} \to X_{\et}$ which is a section of the map $X_{\et} \to (\Spec k)_{\et}$ induced by the structure morphism.  The pro-space $(\Spec k)_{\et}$ is homotopy equivalent to the classifying pro-space $BG$ of $G$. Hence we can consider $X_{\et}$ as an object over $BG$. Since spaces over $BG$ are equivalent, in a sense to be made precise later, to spaces with a $G$-action, we would like to form the homotopy fixed points $X_{\et}^{hG}$ of $X_{\et}$. The \'etale homotopy type functor then induces a natural map from the set $X(k)$ of $k$-rational points to the set of connected components $\pi_0(X_{\et}^{hG})$. The non-existence of homotopy fixed points of $X_{\et}$ would therefore be an obstruction to the existence of rational points. We will show that this idea can also be used to detect rational points.

In order to make this precise we have to specify a suitable model for the \'etale homotopy type. We will use the rigid \v{C}ech \'etale type over $k$ introduced by Friedlander in \cite{fried0}. We denote the resulting pro-space by $\Xh:=(X/k)_{\ret}$. It is weakly equivalent to the usual \'etale topological type of \cite{artinmazur} and \cite{fried}. The pro-space $(\Spec k/k)_{\ret}$ is isomorphic to the classifying pro-space $BG$. Since $G$ is a profinite group, we can consider $BG$ as a profinite space, i.e. an object in the category $\hSh$ of simplicial profinite sets. Moreover, since $X$ is smooth and connected, a result of Artin-Mazur \cite{artinmazur} shows that all its \'etale homotopy groups are profinite groups. This leads to the construction of a concrete fibrant profinite model $\Xhpf$ of the \'etale topological type of $X$ in the category of profinite spaces over $BG$. Via this model we define the continuous \'etale homotopy fixed points $\Xhpf^{hG}$ of $X$ over $k$. 

\begin{remark}
We pause for a moment for the following observation. 
The new model $\Xhpf^{hG}$ for the homotopy fixed point space of $X$ is one of the main technical ingredients of the paper and is a key improvement compared to previous approaches as in \cite{hs} where only a set of connected components of a potential $X^{hG}$ is defined; or as in \cite{gspaces} where a set-theoretic profinite completion process is applied. The new model is based on well-known constructions involving Eilenberg-MacLane spaces and Postnikov towers. In particular, we do not need to apply any kind of (Galois-equivariant) profinite completion functor.  
\end{remark}

We continue the outline of ideas. Taking the rigid \v{C}ech type of the base change $\Xok$ yields a pro-space which we denote by $\bXh$. A nice feature of the rigid \v{C}ech type over $k$ is that the $0$-simplices of $\bXh$ are given by the constant pro-set $X(\ok)$ of $\ok$-valued geometric points. 
Moreover, $\bXh$ inherits a natural action by the absolute Galois group $G$. (One should note that this action is only defined on the whole pro-space and not on each individual space.) 
This induces an action of $G$ on the limit of the underlying diagram of $\bXh$. We denote by $\bXh^G$ the $G$-fixed points of the limit of the inverse system underlying the pro-space $\bXh$.  
The set of $0$-simplices of $\bXh^G$ is then a subset of the set of rational points $X(k)$ of $X$. In particular, we obtain a surjective map of sets 
\[
X(k) \to \pi_0(\bXh^G)
\]
from $X(k)$ to the set of connected components of $\bXh^G$. 
Moreover, there is a canonical map of simplicial sets 
\[
\eta: \bXh^G \to \bXh^{hG}
\]
where we write $\bXh^{hG}$ for the continuous homotopy fixed point space 
\[
\bXh^{hG} := \Xhpf^{hG}
\] 
of $X$ over $k$. Overall we have the following diagram 
\begin{equation}\label{maindiagram}
\xymatrix{
X(k) \ar[rr] \ar[dr] & & \pi_0(\bXh^{hG}) \\
 & \pi_0(\bXh^G) \ar[ur]_{\pi_0(\eta)} & }
\end{equation}
of natural maps of sets. 
Hence if $\pi_0(\eta)$ is surjective, it would follow that each connected component of the homotopy fixed point space $\bXh^{hG}$ corresponds to a rational point of $X$. 

Let us return to the special case of a variety $X$ as in Conjecture \ref{sconj}. It is an example of a $K(\pi,1)$-variety, i.e. its \'etale topological type is weakly equivalent to an Eilenberg-MacLane space of the type $K(\pi,1)$. For such a variety, there is a natural bijection of sets 
\[
\pi_0(\bXh^{hG}) \cong S(\pi_1(X/k))
\]
where we recall that $S(\pi_1(X/k))$ denotes the set of conjugacy classes of continuous sections of \eqref{sesintro}. One should note that for this bijection it is crucial that we are able to define {\it continuous} homotopy fixed points. As a consequence of the previous discussion we can formulate our main result.

\begin{theorem}\label{mainthmintro}
Let $k$ and $X$ be as in Conjecture \ref{sconj}. Then the map $a \mapsto [\sigma_a]$ is surjective {\it if} the map of sets 
\[
\pi_0(\eta): \pi_0(\bXh^G) \to \pi_0(\bXh^{hG})
\]
is surjective. 
\end{theorem}

The question whether the comparison map from fixed points to homotopy fixed points, such as $\eta$, is a weak equivalence is a special case of a homotopy limit problem (see \cite{holimlim}).  Unfortunately, to solve a homotopy limit problem is in general a very difficult task. For example, the comparison of fixed and homotopy fixed points under the action of finite $p$-groups was known as the Sullivan conjecture which has been proved in different variations in the famous works of Miller \cite{miller}, Carlsson \cite{carlsson} and Lannes \cite{lannes}. 
Nevertheless, we are optimistic that in special cases of arithmetic interest there will be enough information on the Galois action to deduce information about $\pi_0(\eta)$ in diagram \eqref{maindiagram}. \\

%It has already been observed  in \cite{pal2} that the collection of such maps is in bijection with the set $S(\pi_1(X/k))$.

The content of the paper is organized as follows. In the second section, we provide a framework for  continuous homotopy fixed points of pro-spaces with an action by a profinite group. In the third section, we discuss rigid \v{C}ech types of algebraic varieties over a field and define a new profinite model for them. In the last section, we define Galois homotopy fixed points of varieties and construct the map $\eta$ which we need for diagram \eqref{maindiagram}. In the final paragraph we show Theorem \ref{mainthmintro}.

{\bf Acknowledgements}: I am grateful to Kirsten Wickelgren and Eric Friedlander for helpful discussions and comments.

\section{Models in profinite homotopy}

\subsection{Notations} 

%For details about the following notions and facts we refer the reader to \cite{ensprofin}, \cite{gspaces} and \cite{gspectra}. 
Let $\Sh$ be the category of simplicial sets whose objects we also call spaces and let $\Shp$ be the category of pointed spaces. We denote by $\hSh$ the category of profinite spaces, i.e. simplicial objects in the category of profinite sets with continuous maps as morphisms. Let $\hShp$ be the associated category of pointed profinite spaces. We consider $\hSh$ and $\hShp$ with the simplicial model structures described in \cite{gspaces} and \cite{gspectra}. (The reader should note that Morel had already introduced the category $\hSh$ and equipped it with a $\Z/p$-model structure in \cite{ensprofin}). 

\begin{example}
Important examples of profinite spaces are classifying spaces for profinite groups. For a profinite group $G$, the simplicial set $BG$ given in degree $n$ by the product of $n$ copies of the profinite group $G$ is in a natural way an object of $\hShp$. Moreover, it comes equipped with the profinite space $EG$ over $BG$, given in degree $n$ by the product of $n+1$ copies of $G$ with a free $G$-action in each dimension. 
\end{example}

If $B$ is a pointed profinite space, we denote by $\hShp/B$ the category of pointed profinite spaces $X$ together with a map $X\to B$ in $\hShp$. This category of profinite spaces over $B$ inherits a model structure from $\hShp$ via the forgetful functor. (One should note that the terminology pointed (profinite) space over $B$ does not require a section $B\to X$ of the structure map $X\to B$.) 

If $X$ and $Y$ are objects in $\hShp/B$, we denote by $\Map_{\hShp/B}(X,Y)$ the simplicial set whose set of $n$-simplices is given as the set maps in $\hShp/B$
\[
\Map_{\hShp/B}(X,Y)_n=\Hom_{\hShp/B}(X \wedge \Delta[n]_+, Y)
\]
is given as the set maps in $\hShp/B$ where the standard simplicial $n$-simplex $\Delta[n]$ is equipped with a disjoint basepoint and the trivial map to $B$. This defines a functor 
\[
\Map_{\hShp/B}(-,-): (\hShp/B)^{\mathrm{op}}\times \hShp/B \to \Sh.
\]

\begin{remark}
Since the model structure on $\hShp/B$ is simplicial (see \cite{gspaces} and \cite[\S 2.2]{gspectra}), this functor is homotopy invariant in the following sense. Let $Z$ be an object in $\hShp/B$ and $f:X\to Y$ a map between fibrant objects in $\hShp/B$. If $f$ is a weak equivalence in $\hShp/B$, then the map $\Map_{\hShp/B}(Z,f)$ is a homotopy equivalence of fibrant simplicial sets. 
\end{remark}

\subsection{Finite models for spaces}

Our first step in the construction of Galois homotopy fixed point spaces is to show that a space with finite homotopy groups has a concrete model in the category $\hSh$.

\begin{defn}\label{pifinitedef}
A connected pointed simplicial set $X$ is called $\pi$-finite if all its homotopy groups are finite. 
%\item $\pi$-profinite if all its homotopy groups are profinite groups and, for each $n\ge 2$, the action of the profinite group $\pi_1(X,x)$ on the profinite abelian group $\pi_n(X)$ is continuous. 
\end{defn}

Starting with a $\pi$-finite space we will show below that it is homotopy equivalent to a profinite space of the following type.

\begin{defn}\label{fspace}
A pointed $f$-space $Z$ is a fibrant pointed profinite space such that each $Z_n$ is a finite set.
\end{defn}

Before we prove the main result of this section, we need the following notations. Let $X$ be a simplicial set. We denote by $\Pi_1X$ the fundamental groupoid of $X$. The higher homotopy groups define a module $\Pi_nX$ over $\Pi_1X$ defined by sending $x\in X_0$ to the $\Pi_1X(x,x)=\pi_1(X,x)$-module $\Pi_nX(x)=\pi_n(X,x)$. 
If $X$ is connected and equipped with a chosen basepoint $x\in X_0$, we choose for each vertex $y \in X_0$ a path $\gamma_y: y \to x$ in $\Pi_1X$ such that $\gamma_x$ is the identity. These paths induce isomorphisms of fundamental groups 
\[
\pi_1(X,y) =\Pi_1X(y,y) \cong_{\gamma} \Pi_1X(x,x)=\pi_1(X,x)=:\pi_1.
\]
Since the higher homotopy groups are abelian, we have canonical isomorphisms $\pi_n(X,y) = \pi_n(X,x)=: \pi_n$. The action of $\Pi_1X$ on $\Pi_nX$ is then determined by the structure of $\pi_n$ as a $\pi_1$-module. We now state the main theorem of this section. A different version of it has been shown in \cite{gspectra}. 

\begin{theorem}\label{finitemodel}
Let $\Gamma$ be a finite group and let $X$ be a connected pointed simplicial set over $B\Gamma$ which is $\pi$-finite and such that $\pi_1(X) \to \Gamma$ is surjective. Then there is a pointed $f$-space $FX$ over $B\Gamma$ which is a fibrant object in $\hShp/B\Gamma$ and a pointed map $\varphi:X \to FX$ over $B\Gamma$ which is a weak equivalence of underlying simplicial sets. In particular, it induces an isomorphism $\pi_*X \cong \pi_*FX$ of homotopy groups of the underlying simplicial sets.  
The assignment $X\mapsto FX$ is functorial for maps between connected simplicial sets over $B\Gamma$ which are $\pi$-finite. 
\end{theorem}
\begin{proof}
After taking a fibrant replacement in $\Shp/B\Gamma$ we can assume that $X$ is a fibrant pointed simplicial set. 
For each $n\geq 1$, let $\pi_n:=\pi_nX$ be the $n$th homotopy group of $X$ which is by assumption a finite group. We construct the pointed profinite space $FX$ as the limit in $\hShp$ of a specific Postnikov tower of $X$
\[
\ldots \to  X(n) \to X(n-1) \to \ldots \to X(1).
\]

Let $\cosk_nX \in \Shp/B\Gamma$ be the $n$th coskeleton of $X$. It comes equipped with  natural maps $X\to \cosk_nX$ and $\cosk_{n}X\to \cosk_{n-1}X$ over $B\Gamma$ for each $n\geq 2$.  
The map $\cosk_nX \to \cosk_{n-1}X$ sits in a homotopy pullback square of space over $B\Gamma$
\[%\begin{equation}\label{cosknpb}
\xymatrix{
\cosk_nX \ar[d] \ar[r] & E\pi_1 \times_{\pi_1} WK(\pi_n,n) \ar[d]^{q_n} \\
\cosk_{n-1}X \ar[r]_{k_n} & E\pi_1 \times_{\pi_1} K(\pi_n,n+1).}
\]%\end{equation} 
The map $q_n$ is induced by the universal bundle over the Eilenberg-MacLane space $K(\pi_n,n+1)$ which we consider as a simplicial finite group. For a simplicial group $\Gh$, the contractible space $W\Gh$ is defined by 
\[
(W\Gh)_n = \Gh_n \times \Gh_{n-1} \times \ldots \times \Gh_0.
\]
The map $k_n$ is the $k$-invariant defined by a class 
\[
[k_n]\in H^{n+1}_{\pi_1}(\cosk_{n-1}X; \pi_n)
\]
in the $\pi_1$-equivariant cohomology of $\cosk_{n-1}X$ (see also \cite[VI \S 5]{gj} and \cite[p. 207-208]{goerss}). 
It fits into a commutative diagram
\[%\begin{equation}\label{kn}
\xymatrix{
\cosk_{n}X \ar[d] \ar[r] & K(\pi_n,n+1) \ar[d] \\
\cosk_{n-1}X \ar[r ]\ar[ur]_{k_n} & B\pi_1.}
\]%\end{equation}

%Since $\pi_1$ and $\pi_n$ are finite groups, the spaces $E\pi_1 \times_{\pi_1} K(\pi_n,n+1)$ and $E\pi_1 \times_{\pi_1} WK(\pi_n,n)$ are simplicial finite sets. Moreover, the map  
%\[
%q_n: E\pi_1 \times_{\pi_1} WK(\pi_n,n) \to E\pi_1 \times_{\pi_1} K(\pi_n,n+1)
%\]
%is a fibration in $\hShg$ by \cite{gspaces}, Theorem 2.9, and \cite{completion} Proposition 3.7. 

Now we define profinite spaces $X(n)$ together with natural maps over $B\Gamma$ 
\[
j_n:\cosk_nX \to X(n)
\] 
which are weak equivalences of underlying simplicial sets. For $n=1$, we define 
\[
X(1):=B\pi_1 \to B\Gamma.
\]
Since $\pi_1 \to \Gamma$ is surjective, this is a fibration in $\hShp$.    
Choosing any map $\cosk_1X \to B\pi_1$ over $B\Gamma$ which is a weak equivalence of underlying simplicial sets provides a map $j_1: \cosk_1X \to X(1)$ over $B\Gamma$. %(This choice is possible, since the category of spaces over $BG$ can be equipped with a model structure in which weak equivalences and fibrations over $BG$ are determined by their underlying maps.)

%Moreover, the action of $\pi_1$ on $\pi_2$ induces a map in $\hShg$ 
%\[
%X(1) = B\pi_1 \to E\pi_1 \times_{\pi_1} K(\pi_2,3).
%\]

For $n\geq 2$, assume we are given $X(n-1)$ and together with a pointed map $j_{n-1}:\cosk_{n-1} \to X(n-1)$ over $B\Gamma$. Up to homotopy, there is a factorization over $B\Gamma$ %in the category of simplicial $G$-sets
\[
\xymatrix{
\cosk_{n-1}X \ar[d]^{k_n} \ar[r]^{j_{n-1}} & X(n-1) \ar[dl] \\
E\pi_1 \times_{\pi_1} K(\pi_n,n+1). & }
\]

The space $X(n)$ and the map $X(n) \to X(n-1)$ is then defined as the pullback of the diagram 
\begin{equation}\label{GXnpb}
\xymatrix{
X(n) \ar[d] \ar[r] &  E\pi_1 \times_{\pi_1} WK(\pi_n,n) \ar[d]^{q_n} \\
X(n-1) \ar[r] & E\pi_1 \times_{\pi_1} K(\pi_n,n+1).}
\end{equation}

Since $\pi_1$ and $\pi_n$ are finite groups, the spaces $E\pi_1 \times_{\pi_1} K(\pi_n,n+1)$ and $E\pi_1 \times_{\pi_1} WK(\pi_n,n)$ are simplicial finite sets. Moreover, the map  
\[
q_n: E\pi_1 \times_{\pi_1} WK(\pi_n,n) \to E\pi_1 \times_{\pi_1} K(\pi_n,n+1)
\]
is a fibration in $\hShp$ by \cite[Theorem 2.9]{gspaces} (or \cite[Proposition 3.7]{completion}). 
Hence the pullback of (\ref{GXnpb}) can be constructed in $\hShp/B\Gamma$, $X(n)$ is a profinite space over $B\Gamma$ which is a fibrant object in $\hShp/B\Gamma$.  
Since the map 
\[
\cosk_nX \to E\pi_1 \times_{\pi_1} WK(\pi_n,n) \times_{E\pi_1 \times_{\pi_1} K(\pi_n,n+1)} \cosk_{n-1}X\] 
is a weak equivalence, we obtain an induced weak equivalence $j_n: \cosk_nX \to X(n)$ of underlying simplicial sets. 
Now we can define $FX$ to be the 
\[
FX:=\lim_n X(n).
\]

Since the set of $m$-simplices of $X(n)$ is isomorphic to the set of $m$-simplices of $X(n-1)$ for $m\leq n-1$, $FX$ is a simplicial object of finite sets. Moreover, $FX\to B\Gamma$ is a fibrant object in $\hShp/B\Gamma$, since it is the filtered inverse limit of a tower of fibrations in $\hShp/B\Gamma$. Furthermore, since the natural maps $X\to \lim_n \cosk_nX$ and $\lim_n \cosk_nX \to \lim_n X(n)$ are weak equivalences of underlying simplicial sets, the associated map $\varphi:X \to FX$ is a weak equivalence of underlying simplicial sets. In particular, it induces an isomorphism $\pi_*X \cong \pi_*FX$. 
The functoriality follows from the fact that all constructions used to define $FX$ are functorial.
\end{proof}
%

%FUNCTORIALITY WITH RESPECT TO CHANGE OF BASE SPACE $BG$ IS PROBABLY NEEDED!!!!\\

%\begin{cor}\label{profinitemodel}
%Let $X$ be a connected pointed simplicial set which is $\pi$-finite. Then there is an $f$-space $FX$ and a map $\varphi:X \to FX$ which is a weak equivalence of underlying simplicial sets. In particular, it induces an isomorphism $\pi_*X \cong \pi_*F_X$ of homotopy groups of the underlying simplicial sets.  The assignment $X\mapsto FX$ is functorial for maps between connected simplicial sets which are $\pi$-finite. 
%\end{cor}

\begin{remark}\label{remarkkpi1profinite}
The construction of the functor $X\mapsto FX$ can be immediately generalized to profinite groups in the following special case. 
Let $G$ be a profinite group and let $(X,x)$ be a connected pointed simplicial set over $BG$ whose only nontrivial homotopy group is the profinite group $\pi_1(X,x)=:\pi$. Then the profinite classifying space $B\pi \in \hShp/BG$
is equipped with a pointed map $X \to B\pi$ over $BG$ which is a weak equivalence of underlying simplicial sets. We consider $B\pi$ as a profinite model for $X$ in $\hShp/BG$. 
\end{remark}

\subsection{Continuous homotopy fixed points}

Let $G$ be a profinite group. We fix a functorial fibrant replacement in $\hShp/BG$ and denote it by $X \mapsto RX$. 

\begin{defn}
For $X\in \hShp/BG$, we define the space $X^{hG}$ to be 
\[%\begin{equation}\label{homotopyfixedpoints}
X^{hG}:=\Map_{\hShp/BG}(BG, RX).
\]%\end{equation}
We call $X^{hG}$ the homotopy fixed point space of $X$. 
\end{defn}

This notation and terminology is justified by the following observation. Let $\hShpg$ be the category of pointed profinite $G$-spaces, i.e. simplicial objects in the category of profinite sets with a continuous $G$-action together with a basepoint which is fixed under $G$. By taking homotopy orbits, we obtain a functor 
\[
\hShpg \to \hShp/BG, ~Y\mapsto (Y\times_G EG \to BG)
\]
from $\hShpg$ to the category of pointed profinite spaces over $BG$. This functor is right adjoint to the functor 
\[
\hShp/BG \to \hShpg, ~X\mapsto X\times_{BG}EG.
\]
Moreover, $Y\mapsto (Y\times_G EG \to BG)$ sends fibrant pointed profinite $G$-spaces to fibrations over $BG$. Let $Y \mapsto R_GY$ be a fixed fibrant replacement in $\hShpg$. Then, for a pointed profinite $G$-space $Y$, we have a natural isomorphism 
\[
\Map_{\hShpg}(EG, R_GY) \cong \Map_{\hShp/BG}(BG, R_GY \times_{G} EG).
\]
The mapping space on the left is the (continuous) homotopy fixed point space of the pointed profinite $G$-space $Y$ (see also \cite{gspaces} and \cite{homfixedlt}).

\begin{remark}
The crucial point in the construction of $X^{hG}$ is that we do take the topology of $G$ into account by considering continuous mapping spaces in $\hShp/BG$. Moreover, the functor $X\mapsto X^{hG}, \hShp/BG \to \Sh$ is homotopy invariant and does not depend on the choice of fibrant replacement in $\hShp/BG$. This follows from the fact that $BG$ is cofibrant in $\hShp/BG$ and that the model structure on $\hShp/BG$ is simplicial (see also \cite{gspaces} and \cite[\S 2.2]{gspectra}).%We refer the reader to \cite{homfixedlt} for more details on continuous homotopy fixed points.
\end{remark}

\subsection{Homotopy fixed points and sections}

For our main arithmetic application, we need to relate homotopy fixed point spaces to the following set of sections. 
Let $\bpi$ be a profinite group and let
\begin{equation}\label{genses}
1 \to \bpi \to \pi \to G \to 1
\end{equation}
be a fixed extension of $G$ by $\bpi$. We denote the set of $\bpi$-conjugacy classes of continuous sections of (\ref{genses}) by $S(\pi)$. 

The homotopy fixed points of the classifying space $B\pi$ are related to the set $S(\pi)$ in the following way.
\begin{prop}\label{profinsection}
There is a natural bijection 
\[
\pi_0(\Map_{\hShp/BG}(B\pi,BG)) \cong S(\pi).
\]
%and a natural isomorphism of groups
%$$\pi_1(B\pi^{hG}) \cong H^0(G;\pi)=\pi^G.$$
\end{prop}
\begin{proof}
The set of connected components of $\Map_{\hShp/BG}(BG, B\pi)$ is in bijection with the set of homotopy classes of maps  
\[
\Hom_{\hHhp/BG}(BG, B\pi).
\]
The universal property of classifying spaces implies that the latter set is in bijection with the set of continuous outer homomorphisms from $G$ to $\pi$ over $G$. The latter set is in bijection with $S(\pi)$. 
\end{proof}

\begin{remark}
One should note that, if the groups $\pi$ and $G$ are infinite profinite groups, it is crucial for the assertion in Proposition \ref{profinsection} that we use mapping spaces in $\hShp/BG$, since we are interested in the set of {\it continuous} sections of \eqref{genses}.
\end{remark}

\subsection{Profinite models for pro-spaces}\label{profinmodels}

Our next goal is to apply the constructions of the previous sections to pro-spaces. Since the \'etale topological type of a variety is given as a pro-object in the category of spaces,  we need this generalization for the arithmetic applications of the next section. 

For a category $\Ch$, let pro-$\Ch$ be the category of pro-objects of $\Ch$, i.e. the category of filtered diagrams in $\Ch$ with morphism sets defined by 
\[
\Hom_{\mathrm{pro-}\Ch}(\{X(i)\},\{Y(j)\}) := \lim_j \colim_i \Hom_{\Ch}(X(i),Y(j)).
\]

%\begin{example}\label{strictexample}
Let $\Xh=\{\Xh(i)\}_I$ and $\Yh=\{\Yh(j)\}_J$ be pro-objects of $\Ch$. Assume we have a functor $\alpha:J \to I$ between the indexing categories and a natural transformation $T: \Xh \circ \alpha \to \Yh$. This datum defines a morphism in pro-$\Ch$ 
\[
(T(j))_{j\in J} \in \lim_j \Hom_{\Ch}(\Xh(\alpha(j)), \Yh(j)) \subset \Hom_{\mathrm{pro-}\Ch}(\Xh,\Yh).
\]
Such a morphism of pro-objects is called a {\it strict} morphism. 
%\end{example}

If $\Ch$ is a simplicial category, then the mapping space of two pro-objects is defined by
\[
\Map_{\mathrm{pro-}\Ch}(\{\Xh(i)\},\{\Yh(j)\}) := \lim_j \colim_i \Map_{\Ch}(\Xh(i),\Yh(j)).
\]

We are interested in the following special situation. 
Let $G=\lim_k G(k)$ be a profinite group given as the inverse limit of finite groups $G(k)$ indexed over the filtered category $K$.  Let $\{\Xh(i)\}_I$ be a pro-object in the category of pointed spaces $\Shp$. We assume that every $\Xh(i)$ is a pointed connected $\pi$-finite space in the sense of Definition \ref{pifinitedef}. Assume that we are given a strict morphism $\{\Xh_i\}_I \to \{BG(k)\}_K$ of pro-objects in $\hShp$. By definition of a strict morphism, this means that we have a functor $\alpha: K \to I$ and natural maps $\Xh(\alpha(k)) \to BG(k)$ in $\hShp$ for every $k \in K$. We assume that for every $k$, the induced homomorphism of fundamental groups $\pi_1(\Xh(\alpha(k)) \to G(k)$ is surjective. 
(For those $i \in I$ for which there might be no $k\in K$ with $\alpha(k)=i$, we consider $X(i)$ to be a pointed space over the trivial classifying space $B\{e\}=*$.)

Now we apply the functor $F: \Xh(i) \mapsto F\Xh(i)$ of Theorem \ref{finitemodel} to each $i\in I$. We obtain a pro-object $\{F\Xh(i)\}_I$ in the category of pointed $f$-spaces in the sense of Definition \ref{fspace} together with a strict morphism 
\[
\{F\Xh(i)\}_I \to \{BG(k)\}_K
\]
of pro-objects in $\hShp$. Since taking homotopy limits is functorial with respect to strict morphisms, we get an induced map in $\hShp$
\[
\varphi: \holim_i F\Xh(i) \to \holim_k BG(k)
\]
which, by abuse of notations, is also denoted by $\varphi$. Since filtered homotopy inverse limits preserve fibrations, $\varphi$ is a fibration in $\hShp$. 
(We refer the reader to \cite[\S 2.5]{gspectra} for homotopy limits in $\hShp$.)

%\begin{defn}
%We call $\Xhpf:=\holim_i F\Xh(i)$ together with the map $\varphi$ a profinite model of the pro-space $\{\Xh(i)\}_I$. 
%\end{defn}
%We define the profinite model of the pro-space $\{X(i)\}_I$ to be the homotopy limit 
%\[
%\holim_i FX(i) \in \hShp
%\]
%in the category of pointed profinite spaces (see \cite[\S 2.5]{gspectra} for homotopy limits in $\hShp$). 

\begin{lemma}\label{profinhomotopygroups}
For each $n\ge 0$, the homotopy group $\pi_n(\holim_i F\Xh(i))$ is naturally isomorphic in the category of profinite groups to the profinite group $\{\pi_n(\Xh(i))\}_I$. 
\end{lemma}
\begin{proof}
By Theorem \ref{finitemodel}, we have natural isomorphisms $\pi_n(\Xh(i))\cong \pi_n(F\Xh(i))$ for every $i \in I$. Since the category of profinite groups is canonically equivalent to the pro-category of finite groups, it suffices to show that the homotopy group $\pi_n(\holim_i F\Xh(i))$ is isomorphic to the profinite group $\lim_i \pi_n(\Xh(i))$. But this follows as in \cite[Lemma 2.14]{gspectra} from the Bousfield-Kan spectral sequence for homotopy limits. 
\end{proof}

The previous lemma justifies the following terminology.
\begin{defn}
We call $\Xhpf:=\holim_i F\Xh(i) \in \hShp$ together with the map $\varphi$ to $BG$ in $\hShp$ a profinite model over $BG$ of the pro-space $\Xh=\{\Xh(i)\}_I$. 

We define the continuous $G$-homotopy fixed points of $\Xh$ to be the space 
\[
\Xhpf^{hG} :=\Map_{\hShp/BG}(BG, \Xhpf).
\]
\end{defn}

\begin{remark}
The canonical map from limits to homotopy limits induces a natural map of underlying pointed spaces
\begin{equation}\label{limholim}
\lim_i \Xh(i) \to \holim_iF\Xh(i).
\end{equation}
\end{remark}
%\[
%\begin{array}{rcl}
%\Map_{\mathrm{pro-}\Shp}(*,\{X(i)\}) & = & \lim_i \Map_{\Shp}(*, X(i))\\
%& \to & \lim_i \Map_{\hShp}(*, FX(i)) \\
%& \to & \holim_i \Map_{\hShp}(*, FX(i)) \\
%& \cong & \Map_{\hShp}(*, \holim_i FX(i)).
%\end{array}
%\]

\begin{remark}\label{remarkpi1profinitemodel}

%Let $\Gamma$ be a group. Recall that our functorial model for the classifying space $B\Gamma$ is the simplicial  set which is in degree $n$ given by the $n$-fold product of copies of $\pi$. If $\Gamma$ is profinite, then $B\Gamma$ is an object in $\hSh$, and even in $\hShp$, since it has only one vertex. 
 
In the above situation, let us assume that $\{\Xh(i)\}_I$ be a pro-space such that each $\Xh(i)$ is a pointed connected $\pi$-finite space whose only nontrivial homotopy group is $\pi_1$. Then $F\Xh(i)$ is just given by $B\pi_1(\Xh(i))$. The limit $\lim_i B\pi_1(\Xh(i))$ is isomorphic to the simplicial profinite set $B(\lim_i\pi_1(\Xh(i)))$ which in degree $n$ is given by the $n$-fold product of copies of the profinite group $\lim_i\pi_1(\Xh(i))$. The canonical map
\[
\lim_i B(\pi_1(\Xh(i))) \to \holim_i B(\pi_1(\Xh(i)))
\]
is then a weak equivalence of pointed profinite spaces. Hence in this case, 
\[
B(\lim_i\pi_1(\Xh(i))) = \lim_i B(\pi_1(\Xh(i)))  \to \lim_k BG(k) = BG
\] 
would just as well serve as a profinite model of the pro-space $\{\Xh(i)\}_i$.
%
%For each $i$, the functorial finite replacement of $X_i$ is just given by a map $X_i \to B\pi_1(X_i)$ which is a weak equivalence of underlying simplicial sets.
\end{remark}

\subsection{Group actions on pro-spaces}\label{proaction}

Finally, the Galois action on the \'etale topological type of a variety leads us to the following notion of a group action on a pro-space.  

Let $G$ be a profinite group and $\Xh=\{\Xh(i)\}_I$ be as above. Let $\Yh=\{\Yh(j)\}_J$ be another pro-object of $\Sh$. We assume that $G$ acts on $\Yh$ in the sense that every element $g\in G$ induces a strict automorphism of $\Yh$. 
%Here we do not assume that the basepoints of the individual space $\Yh(j)$ are fixed under $G$. We just assume that the basepoint of $\Yh(g(j))$ is sent to the one of $\Yh(j)$ under the transformation corresponding to $g\in G$. %, i.e. $g$ defines a functor $g^*: I \to I$ and is equipped with a natural transformation $X \circ g^* \to X$. 
Then the abstract group $G$ acts on the mapping space 
\[
\Map_{\mathrm{pro-}\Sh}(*,\{\Yh(j)\}) = \lim_j \Yh(j)
\] 
as well. 

Now let $f: \Yh\to \Xh$ be a strict morphism of pro-objects form $\Yh$ to a pro-object $\Xh$ over $BG$ which satisfies the hypotheses of the previous section \ref{profinmodels}. Then $f$ induces a natural map 
\[
\lim_j \Yh(j) \to \lim_i \Xh(i).
\] 
In particular, we have a map 
\[
(\lim_j\Yh(j))^G \to \lim_i \Xh(i)
\]
from the $G$-fixed points in $\lim_j\Yh(j)$. After taking a profinite model for $\{\Xh(i)\}_I$ as above, we obtain via \eqref{limholim} a natural map of spaces 
\[
\eta: (\lim_j\Yh(j))^G \to \lim_i \Xh(i) \to \holim_i F\Xh(i) \to \Map_{\hShp/BG}(BG, \Xhpf) = \Xhpf^{hG}.
\]

\begin{remark}
In the case that $\Xh$ has the homotopy type of the homotopy orbit space of $\Yh$ under its $G$-action, we may consider $\Xhpf^{hG}$ as the continuous homotopy fixed points of $\Yh$ and also write 
\[
\Yh^{hG}:= \Xhpf^{hG}
\] 
for this space. Moreover, we then consider $\eta$ as a map from fixed points to the homotopy fixed points of $\Yh$ under $G$. The main example for this situation is the Galois action on the \'etale topological type of a smooth variety over a field. 
\end{remark}

\section{\'Etale topological types}

We will now turn to the cases of arithmetic geometric origin in which we apply the ideas of the previous sections. The first step is to choose a specific model for the \'etale topological type of a variety. Instead of using the \'etale type of schemes defined by Friedlander in \cite{fried}, we consider the rigid \v{C}ech \'etale topological type over a field. It has been first defined and applied by Friedlander in \cite{fried0}.

\subsection{Rigid \v{C}ech types over a field}

We briefly recall the definition of the rigid \v{C}ech type of a variety over a field from \cite[\S 3]{fried0}. We start with the notion of a rigid covering. 
Let $k$ be a field, $\ok$ an algebraic closure of $k$ and let $X$ be a scheme of finite type over $k$. We denote by $X(\ok)$ the set of geometric points of $X$ with values in $\ok$ covering the structure morphism $p:X \to \Spec k$. 
A rigid covering $\alpha: U\to X$ of $X$ over $k$ is a disjoint union of pointed, \'etale, separated maps
\[
\coprod_{x\in X(\ok)} (\alpha_x: U_x, u_x \to X,x)
\]
where each $U_x$ is connected and $u_x$ is a geometric point of $U_x$ such that $\alpha_x \circ u_x =x$. If $Y$ is another scheme of finite type over $k$ and $f:X\to Y$ is a morphism of schemes, then a morphism of rigid coverings $\phi:(\alpha:U\to X) \to (\beta:V\to Y)$ over $f$ is a morphism of schemes $\phi:U\to V$ over $f$ such that $\phi\circ u_x=v_{f(x)}$ for all $x\in X(\ok)$.

If $\alpha:U\to X$ and $\beta: V\to Y$ are rigid coverings of $X$ and $Y$ over $k$, then the rigid product $U \stackrel{R}{\times}_k V \to X\times_k Y$ is defined to be the closed and open immersion of $U\times_k V\to X\times_k Y$ given as the disjoint union indexed by geometric points $x\times y$ of $X\times_k Y$ of 
\[
\alpha_x \times \beta_x: (U_x\times_k V_y)_0 \to X \times_k Y
\]
where $(U_x \times_k V_y)_0$ is the connected component of $U_x \times_k V_y$ containing the distinguished geometric point $u_x\times v_y$.

If $f:X\to Y$ is a map of schemes and $V\to Y$ a rigid covering of $Y$, then the pullback $f^*(V\to Y)=U\to X$ is the disjoint union of pointed maps 
\[
(V_{f(x)}\times_Y X)_x \to X
\]
where $(V_{f(x)}\times_Y X)_x$ is the connected component of $V_{f(x)}\times_Y X$ containing the geometric point $f(x)\times x$.

The category of rigid coverings of $X$ over $k$ is denoted by $RC(X/k)$. 
The fact that each connected component $U_x$ is equipped with a geometric point implies that there is at most one map between any two objects of $RC(X/k)$. For, a map of connected, separated \'etale schemes over $X$ is determined by the image of any geometric point (see \cite[Proposition 4.1]{fried}). Together with the construction of rigid products this shows that $RC(X/k)$ is essentially a directed set.

For a rigid covering $U\to X$, we denote by $N_X(U)=\cosk_0^X(U)$ its \v{C}ech nerve, i.e. the simplicial scheme given in degree $n$ by the $(n+1)$-fold fiber product of $U$ with itself over $X$. Since $X$ is locally noetherian, the connected component functor $\pi$ is well-defined. In \cite[\S 3]{fried0}, Friedlander defines the rigid \v{C}ech \'etale topological type of $X$ over $k$ to be the pro-simplicial set 
\[
(X/k)_{\ret}:RC(X/k) \to \Sh
\]
given by sending $U\to X$ in $RC(X/k)$ to the simplicial set $\pi(N_X(U))$ of connected components  of the \v{C}ech nerve. 
For a map $f:X\to Y$ of schemes of finite type over $k$, there is a strict morphism
\[
f_{\ret}: (X/k)_{\ret} \to (Y/k)_{\ret}
\]
of pro-simplicial sets induced by the pullback functor $f^*:RC(Y/k) \to RC(X/k)$. This makes the assignment 
\[
X \mapsto (X/k)_{\ret}
\] 
into a functor from the category of schemes of finite type over $k$ to the category of pro-simplicial sets.

The following proposition shows that if $X$ is quasi-projective, then $(X/k)_{\ret}$ has the same homotopy type as the usual \'etale topological type. The proof follows from a combination of Friedlander's arguments in \cite[Proposition 3.2 and a remark on page 102]{fried0}, and \cite[Proposition 8.2]{fried}.

\begin{prop}\label{whe}
Let $X$ be a quasi-projective scheme of finite type over a field $k$. Then there is a zig-zag of canonical weak equivalences in $\pro-\Sh$ between $(X/k)_{\ret}$ and the \'etale topological type $X_{\et}$ of \cite[\S 4]{fried}. 
\end{prop}

%\begin{remark}
%The zig-zag of weak equivalences in the previous proposition is given as follows. First
%(By  choosing an algebraic closure $\overline{k(x)}$ for the residue field $k(x)$ of each scheme theoretic point $x$ of $X$ and by only considering geometric points of the form $\Spec \overline{k(x)} \to X$, on can ensure that $\oX$ is a set.) 
%\end{remark}

\begin{remark}\label{adv} 
The set of $0$-simplices of $\pi(N_X(U))$ for any rigid cover $U\to X$ in $RC(X/k)$ is the set $X(\ok)$ of geometric points with values in $\ok$. Hence the pro-set of vertices of $(X/k)_{\ret}$ is just the constant functor sending each rigid covering $U=\coprod_{x\in X(\ok)}U_x \to X$ to $X(\ok)$.
This makes $(X/k)_{\ret}$ a very convenient object for our purposes.  
\end{remark}

%For the \'etale type of a field, we have the following well-known fact. 

\begin{lemma}\label{BGk}
Let $k$ be a field with absolute Galois group $G$. The rigid \'etale \v{C}ech type of $k$ is isomorphic in $\pro-\Sh$ to the pro-classifying space $BG$, i.e. there is an isomorphism 
\[
(\Spec k/k)_{\ret} \cong BG.
\]
\end{lemma}
\begin{proof}
Let $L/k$ be a finite Galois extension of $k$ contained in a fixed separable closure $\ok$. The associated \v{C}ech nerve $N_k(L)$ consists in degree $n$ of the fiber product over $\Spec k$ of $n+1$ copies of $\Spec L$. The set of connected components in each degree is hence just given by the product of $n$ copies of the finite Galois group $\Gal(L/k)$ of the extension $L/k$. Hence the simplicial set of connected components of $N_k(L)$ is naturally isomorphic to $B\Gal(L/k)$. Since every rigid cover defining $(\Spec k/k)_{\ret}$ is given by a finite Galois extension $L\subset \ok$, this proves the assertion.
\end{proof}

The following two examples of morphisms will be most important for us.

\begin{example}\label{structuremap}
Let $X$ be a geometrically connected variety over $k$. The map $p_{\ret}: (X/k)_{\ret} \to (\Spec k/k)_{\ret}$ induced by the structure map $p:X\to \Spec k$ has the following shape. As we have mentioned in the previous proof, a rigid cover of $\Spec k$ is given by a finite Galois extension $L/k$ inside the chosen algebraic closure $\ok$. The pullback functor $p^*: RC(k/k) \to RC(X/k)$ sends the finite Galois extension $L/k$ to the rigid cover $U_L \to X$ 
\[
(U_L \to X):= \coprod_{x\in X(\ok)} X_L,x_L \to X, x \in RC(X/k).
\] 
given by the disjoint union of the (connected) finite Galois covers $X_L=X\times_k L \to X$ indexed by the geometric points $x\in X(\ok)$.  
The component $X_L$ is equipped with the canonical lift $x_L$ of $x$ induced by the map $\Spec \ok \to \Spec L$. The canonical isomorphism $X_L\times_X X_L=X\times_k(L\times_k L)$ 
%is isomorphic to a disjoint union of $\#(\Gal(L/k))$ many copies of $X$, 
induces a functorial map of simplicial sets 
\[
\pi(N_X(U_L)) \to \pi(N_k(L)).
\]
%given in each dimension by sending the connected component $X$ corresponding to an element in $\pi(N_k(L))$ to this element. 
This determines the strict map $p_{\ret}$ as an element in the set 
\[
\lim_{L/k} \Hom_{\Sh}(\pi(N_X(U_L)), \pi(N_k(L))). 
\] 
In particular, since $\pi(N_k(L))$ is isomorphic to $B\Gal(L/k)$, we see that each simplicial set $\pi(N_X(U_L))$ is equipped with a map to the classifying space $B\Gal(L/k)$ of the finite group $\Gal(L/k)$. Since $X$ is geometrically connected over $k$, this map induces a surjective homomorphism of fundamental groups. 
\end{example}

\begin{example}\label{0action}
Let $X$ be a geometrically connected variety over $k$. Every  element $g\in \Gal(\ok/k)$ defines a morphism $\Xok \to \Xok$ of $\Xok=X\otimes_k \ok$. The induced map $g_{\ret}:(\Xok/k)_{\ret} \to (\Xok/k)_{\ret}$ of rigid \'etale types is induced by the functor $g^*:RC(\Xok/k)\to RC(\Xok/k)$ sending the rigid cover  
\[
\coprod_{x\in X(\ok)} U_{x} \to \Xok
\]
to the rigid cover 
\[
\coprod_{x\in X(\ok)} (U_{g(x)}\times_{\Xok} \Xok)_{x} \to \Xok
\]
where $U_{g(x)}\times_{\Xok} \Xok$ is the fiber product of the diagram
\[
\xymatrix{
U_{g(x)}\times_{\Xok} \Xok \ar[r] \ar[d] & U_{g(x)} \ar[d]\\
\Xok \ar[r]_g & \Xok}
\]
and $(U_{g(x)}\times_{\Xok} \Xok)_{x}$ is the connected component containing $x$. 
%Under this map, the component $U_x$ is sent to 
Hence on $0$-simplices, the map $g_{\ret}$ is given by sending the connected component $(U_{g(x)}\times_{\Xok} \Xok)_{x}$ indexed by $x$ to the component $U_{g(x)}$ indexed by $g(x)$. 

We conclude that, after identifying the pro-set of $0$-simplices with the set of geometric points $X(\ok)$ over $\ok$, the map $g_{\ret}$ is just given by the natural action of $g$ on $X(\ok)$. Moreover, a $0$-simplex of $(\Xok/k)_{\ret}$ which is fixed under the action of all elements $g\in \Gal(\ok/k)$ must be indexed by a rational point of $X$. 
\end{example}

\subsection{Profinite models for \'etale types}\label{etaleprofinmodels}

Let $k$ be a field with algebraic closure $\ok$ and absolute Galois group $G:=\Gal(\ok/k)$. Let $X$ be a geometrically connected smooth variety over $k$. In the following we denote the rigid \v{C}ech type $(X/k)_{\ret}$ of $X$ over $k$ by $\Xh$ and write $I$ for the indexing category $RC(X/k)$, i.e. $\Xh=\{\Xh(i)\}_I$.  

Let $x:\Spec \ok \to X$ be a geometric point of $X$. This turns $\Xh$ into a pro-object of $\Shp$. By Lemma \ref{BGk}, we can identify pro-spaces over $(\Spec k/k)_{\ret}$ with pro-spaces over $BG$. Hence we can consider $\Xh$ as a pro-space over $BG$. By our assumption on $X$ and by \cite[Theorem 11.1]{artinmazur}, every $\Xh(i)$ is a pointed connected $\pi$-finite space in the sense of Definition \ref{pifinitedef}. Moreover, by Example \ref{structuremap}, we know that each $\Xh(i)$ is equipped with a map to the classifying space $B\Gamma$ for some finite quotient group $\Gamma$ of $G$. Since $X$ is a geometrically connected variety over $k$, we can assume that each of the maps $\Xh(i) \to B\Gamma$ induces a surjective homomorphism $\pi_1(\Xh(i)) \to \Gamma$.

Hence, as described in Section \ref{profinmodels}, we can associate to $\Xh$ a functorial profinite model $\Xhpf$ over $BG$. By Lemma \ref{profinhomotopygroups} and Proposition \ref{whe}, we obtain the following result. 

\begin{theorem}
For $k$ and $X$ as above, the fibrant profinite space $\Xhpf$ over $BG$ has the same homotopy type as the \'etale topological type of $X$, i.e. there is an isomorphism of profinite groups $\pi_n(\Xhpf) \cong \pi_n^{\et}(X)$ for all $n\geq 1$. 
\end{theorem} 

%The rigid \v{C}ech type $(\Xok/\ok)_{\ret}$ of the base change to the algebraic closure is denoted by $\bXh$. The pro-space $\bXh$ is equipped with a natural action of $G$. 

%By \cite[Theorem 1.1]{gspaces}, the profinite model of $\Xh$ over $BG$ is has the homotopy type of the $G$-homotopy orbits of $\bXh$. 

\section{Rational points and Galois homotopy fixed points}

We can now give a new definition of continuous homotopy fixed points of a smooth variety over a field under the natural Galois action. A previous definition has been given in \cite{gspaces}. In \cite{hs}, Harpaz and Schlank provide a definition only for the set of connected components of a potential homotopy fixed point space. 

\subsection{Galois homotopy fixed point spaces}

Let $k$ be a field with algebraic closure $\ok$ and absolute Galois group $G:=\Gal(\ok/k)$. Let $X$ be a geometrically connected smooth variety over $k$ and $\Xok$ be its lift to $\ok$. We denote $(X/k)_{\ret}$ by $\Xh=\{\Xh(i)\}_I$ and write $\bXh=\{\bXh(\bi)\}_{\bI}$ for the rigid \v{C}ech type $(\Xok/k)_{\ret}$ of $\Xok$. The pro-space $\bXh$ is equipped with a natural action of $G$ of the form described in Section \ref{proaction}. 

Let $x:\Spec \ok \to X$ be any geometric point of $X$. It turns $\Xh$ into a pro-object in $\Shp$. In particular, we can form the profinite model $\Xhpf$ over $BG$ of $\Xh$ described in Section \ref{etaleprofinmodels}. Essentially the same proof as for \cite[Theorem 3.5]{gspaces} shows that the profinite model of $\Xh$ over $BG$ has the homotopy type of the $G$-homotopy orbits of $\bXh$. 

\begin{defn}
We define
\[
\bXh^{hG}: = \Xhpf^{hG} = \Map_{\hShp/BG}(BG, \Xhpf)
\] 
to be the continuous homotopy fixed point space of $\bXh=(\Xok/k)_{\ret}$. %We denote the space $\Xhpf^{hG}$ also by $\bXh^{hG}$.
\end{defn}

The canonical morphism $\Xok \to X$ induces a morphism of pro-objects of pointed spaces $\bXh \to \Xh$. As explained in Section \ref{proaction}, this yields a canonical map
\begin{equation}\label{geometriceta}
\eta: \bXh^G \to \bXh^{hG}
\end{equation}
from the $G$-fixed points $\bXh^G=(\lim_{\bi} \bXh(\bi))^G$ to the continuous homotopy fixed points space of $\bXh$.

\begin{remark}
One should note that the action of $G$ on $(\Xok/k)_{\ret}$ is only defined on the whole pro-object and not on each space. Hence $(\Xok/k)_{\ret}$ is in general not a pro-object of simplicial $G$-sets. But after forming the mapping space, i.e. after taking the limit of the underlying filtered diagram, we obtain a simplicial object in the category of $G$-sets as described in Section \ref{proaction}.
\end{remark}

\subsection{Rational points and homotopy fixed points}

We keep the notations of the previous section. % and still assume that $X$ is pointed by an arbitrary geometric point. 
By functoriality of rigid \v{C}ech types, every rational point of $X$ induces a map of pro-spaces 
\[
(\Spec k/k)_{\ret} \to \Xh
\]
compatible with the induced structure map $\Xh \to (\Spec k/k)_{\ret}$. %By Lemma \ref{BGk}, we can identify $(\Spec k/k)_{\ret}$ with $BG$. 
After taking profinite models, we get a well-defined map of sets
\begin{equation}\label{rationaltohompoints1}
X(k) \to \Hom_{\hHhp/BG}(BG, \Xhpf) \cong \pi_0(\Map_{\hShp/BG}(BG, \Xhpf)).
\end{equation}

Since the mapping space on the right of \eqref{rationaltohompoints1} is by definition the continuous homotopy fixed point space of $\Xhpf$ (which we also denote by $\bXh^{hG}$), we obtain a natural map of sets 
\[%\begin{equation}\label{rationaltohompoints2}
X(k) \to \pi_0(\bXh^{hG}).
\]%\end{equation}

Moreover, we know from Remark \ref{adv} that the pro-set of $0$-simplices of $\bXh$ is canonically isomorphic to the set $X(\ok)$ of $\ok$-valued geometric points of $X$. By Example \ref{0action}, the action of the absolute Galois group $G$ of $k$ on the pro-set of $0$-simplices of $\bXh$ is just given by the natural action of $G$ on $X(\ok)$. Since each fixed point under this $G$-action has to be indexed by a rational point, we see that the set of $G$-fixed points the $0$-simplices of $\bXh$ is a subset of the $k$-rational points $X(k)=X(\ok)^G$ of $X$. Hence we obtain a canonical surjective map of sets 
\[%\begin{equation}\label{rationaltofixedmap}
X(k) \to \pi_0(\bXh^G).
\]%\end{equation}

Together with the map $\eta$ we obtain a map of sets
\[%\begin{equation}\label{rationaltohomfixedmapfactored}
X(k) \to \pi_0(\bXh^G) \xrightarrow{\pi_0(\eta)} \pi_0(\bXh^{hG}).
\]%\end{equation}

Hence, since the map $X(k) \to \pi_0(\bXh^G)$ is surjective, it is possible to detect rational points on the smooth $k$-variety $X$ by studying the map $\eta: \bXh^G \to \bXh^{hG}$ which we consider as a fixed points to homotopy fixed points map. 

\begin{remark}
It is important that we are able to consider {\it continuous} homotopy fixed points under the action of the profinite Galois group. One reason will be given in the final section where we will see that, for a suitable $X$, the set $\pi_0(\Map_{\hShp/BG}(BG, \Xhpf))$ is in bijection with the set of continuous sections of the short exact sequence \eqref{sesintro} of the introduction. Another reason is given by the following argument. The $E_2$-terms of a descent spectral sequence for Galois homotopy fixed points should be isomorphic to Galois cohomology and not to ordinary group cohomology. This is in fact the case for our definition of $\Xhpf^{hG}$. If we are given a rational point $x$ on $X$, there is a conditionally convergent spectral sequence of the form
\[
E_2^{s,t}=H^s(G;\pi_t(\bXh,x)) \Rightarrow \pi_{t-s}(\Xhpf^{hG})
\]
where $H^s(G;\pi_t(\bXh))$ denotes the continuous cohomology of $G$ with coefficients in the profinite $G$-module $\pi_t(\bXh,x)$ (respectively profinite $G$-set for $t=0$ and profinite $G$-group if $t=1$). A proof of this statement can be read off from the arguments given in the proofs of \cite[Theorem 2.16]{gspaces} and \cite[Theorem 3.17]{homfixedlt}.
\end{remark}

\subsection{The section conjecture as a homotopy limit problem}

Our main example of a case where this observation might be interesting is Grothendieck's section conjecture. 
Let $k$ be a field finitely generated over $\Q$ and $G=\Gal(\ok/k)$ its absolute Galois group. Let $X$ be a geometrically connected variety over $k$. For any given geometric point $x$ of $X$, there is a natural short exact sequence of \'etale fundamental groups 
\begin{equation}\label{ses}
1 \to \pi_1(\Xok,x) \to \pi_1(X,x) \to G \to 1.
\end{equation}

Let $a: \Spec k \to X$ be a rational point on $X$ and let $y: \Spec \ok \to X$ be a geometric point lying above $a$. Applying the functor $\pi_1(-, y)$ to the morphism $a$ induces a continuous homomorphism of groups  
\[
\sigma_a:G \to \pi_1(X, y).
\]
Since $X$ is geometrically connected, there is an \'etale path from $y$ to $x$ which induces an isomorphism $\lambda: \pi_1(X, y) \to \pi_1(X, x)$. Composing $\sigma_a$ with $\lambda$ defines a section 
\[
\lambda \circ \sigma_a:G \to \pi_1(X, x)
\]
of sequence (\ref{ses}). The choice of a different path from $y$ to $x$ changes this section by composition with an inner automorphism of $\pi_1(\Xok, x)$. Hence a rational point of $X$ induces a section of (\ref{ses}) which is well-defined up to conjugacy by an element of $\pi_1(\Xok, x)$. We denote the conjugacy class of the section induced by the rational point $a$ by $[\sigma_a]$ and denote the set of all $\pi_1(\Xok, x)$-conjugacy classes of sections of (\ref{ses}) by $S(\pi_1(X/k))$. With these notations, there is a map of sets 
\begin{equation}\label{scmap}
X(k) \to S(\pi_1(X/k)), ~ a \mapsto [\sigma_a].
\end{equation}

Grothendieck's section conjecture states that map \eqref{scmap} is a bijection if $X$ is a smooth projective curve of genus at least two. It is known that the map is injective. The harder and still open question is whether it is surjective.

We would like to shed some light on map \eqref{scmap} from an \'etale homotopy-theoretical point of view. The crucial and well-known observation is that $X$ is a $K(\pi,1)$-variety over $k$ (see for example \cite{stixbook}). Denoting the rigid \v{C}ech type $(X/k)_{\ret}$ of $X$ again by $\Xh=\{\Xh(i)\}_I$, we know that each $\Xh(i)$ is a pointed connected $\pi$-finite space whose only nontrivial homotopy group is the fundamental group $\pi_1(\Xh(i))$. The pro-system of these finite fundamental groups is just the profinite \'etale fundamental group $\pi_1(X)=\pi^{\et}_1(X, x)$ of $X$. Hence there is a weak equivalence of pro-spaces
\[%\begin{equation}\label{cechbpi}
\Xh \simeq B\pi_1(X).
\]%\end{equation}
Moreover, we can take $\holim_i B\pi_1(\Xh(i))$, or equivalently, as we explained in Remark \ref{remarkpi1profinitemodel}, $\lim_i B\pi_1(\Xh(i)) = B\pi_1(X)$, as a fibrant profinite model $\Xhpf$ of $\Xh$ in $\hShp/BG$.  

%there is a weak equivalence of pro-spaces
%\begin{equation}\label{cechbpiok}
%(\Xok/\ok)_{\ret} \simeq B\pi_1(\Xok, \bar{x}).
%\end{equation}

%\begin{remark}
%It would be nice to express (\ref{cechbpiok}) as a weak equivalence of pro-$G$-spaces. This would require that $G$ acts on $\pi_1(\Xok, x)$. But this can only be the case if $\bar{x}$ is fixed by $G$, i.e. if $\bar{x}$ is a rational point of $X$. Since the surjectivity of map (\ref{scmap}) is exactly concerned with the existence of rational points, it is crucial to avoid to assume the existence of a rational point. 
%\end{remark}

We denote the rigid \v{C}ech type $(\Xok/k)_{\ret}$ again by $\bXh$. Then we have the canonical map of sets 
\[%\begin{equation}\label{rationaltohomfixedmapfactored2}
X(k) \to \pi_0(\bXh^G) \xrightarrow{\pi_0(\eta)} \pi_0(\bXh^{hG})
\]%\end{equation}
described in the previous section. 

Furthermore, we deduce from Proposition \ref{profinsection} that there is a natural bijection of sets 
\[
\pi_0(\bXh^{hG}) \cong S(\pi_1(X/k)).
\]

It follows from this result that map (\ref{scmap}) is surjective if the map
\[%\begin{equation}\label{scashf}
X(k)\to \pi_0(\bXh^{hG})
\]%\end{equation}
is surjective. Since the map $X(k) \to \pi_0(\bXh^G)$ is surjective, we get the following criterion.

\begin{theorem}\label{mainthm1}
Let $k$ be a field which is finitely generated over $\Q$ and let $X$ be a smooth, projective curve of genus $g\geq 2$. Then the map (\ref{scmap}), $a \mapsto [\sigma_a]$, is surjective if the map 
\[
\pi_0(\bXh^G) \xrightarrow{\pi_0(\eta)} \pi_0(\bXh^{hG})
\]
is surjective. 
\end{theorem}

%\begin{remark}
%If there is a $G$-invariant geometric point, i.e. a $k$-rational point $x$ on $X$, then $G$ acts continuously on the fundamental group $\pi_1(\Xok, x)$. In this case we can determine the homotopy type of each of the connected components of the homotopy fixed point space $\bXh^{hG}$ pointed at the image of $x$. This idea already appeared in \cite{stixbook}, \S 9.3.

%By \cite{gspaces}, Theorem 2.16, there is a convergent descent spectral sequence 
%$$E_2^{s,t}=H^s(G;\pi_t(\Xok,x)) \Rightarrow \pi_{t-s}(\bXh^{hG},x)$$
%starting from continuous cohomology of $G$ with coefficients in the profinite $G$-group $\pi_1(\Xok,x)$. This is a fringed second quadrant spectral sequence whose differentials have degree $(r,r-1)$ and such that $E_r^{s,t}$ is not defined for $t-s<0$. 

%Since $(\Xok/\ok)_{\hret}$ is weakly equivalent to $B\pi_t(\Xok,x)$, the spectral sequence degenerates at the $E_2$-stage. The only non vanishing $E_2$-terms are the cohomology groups $H^s(G;\pi_1(\Xok,x))$ for $s=0,1$. Moreover, we obtain an isomorphism 
%$$\pi_{1-s}((\Xok/\ok)_{\hret}^{hG},x) \cong H^s(G;\pi_1(\Xok,x)).$$ 

%We have determined $\pi_{0}((\Xok/\ok)_{\hret}^{hG})$ as $\Sh(\pi_1(X/k))$ in Proposition \ref{profinsection}. The spectral sequence shows that $\pi_1((\Xok/\ok)_{\hret}^{hG})$ is isomorphic to $H^0(G;\pi_1(\Xok,x))$. But this group vanishes as well by \cite{stixbook}, Proposition 79. Hence the space $(\Xok/\ok)_{\hret}^{hG}$ is weakly equivalent to a collection of contractible components which are in bijection with $\Sh(\pi_1(\Xok,x))$. 
%\end{remark}

\begin{remark}\label{propversion}
The category $\hSh$ of profinite spaces has first been studied by Morel in \cite{ensprofin} where a model structure was constructed in which the weak equivalences are the maps that induce an isomorphism in continuous $\Z/p$-cohomology. Since it seems more likely that techniques from the proofs of the Sullivan conjecture (\cite{miller}, \cite{carlsson}, \cite{lannes}) can be translated first to the pro-$p$-case, one may consider it to be a more accessible problem to decide one of the following related questions. Is a $p$-completed version of $\eta$ a weak equivalence? Does $\eta$ induce an isomorphism on mod $p$-homology?
\end{remark}

\bibliographystyle{amsplain}

\begin{thebibliography}{999999}
%
\bibitem{artinmazur} M.\,Artin, B.\,Mazur, Etale homotopy, Lect. Notes in Math., vol. 100, Springer, 1969.
%
\bibitem{ah} A.\,Asok, C.\,Haesemeyer, Stable $\A^1$-homotopy and $R$-equivalence, J. Pure Appl. Algebra 215 (2011), 2469-2472.
%
%\bibitem{bourb} N.~Bourbaki, Topologie G\'en\'erale, Hermann, 1971.
%
%\bibitem{bouskan} A.K.~Bousfield, D.M.~Kan, Homotopy limits, Completions and Localizations, Lecture Notes in Mathematics, vol. 304, Springer-Verlag, 1972.
%
\bibitem{carlsson} G.\,Carlsson, Equivariant stable homotopy and Sullivan's conjecture, 
Invent. Math. 103 (1991), no. 3, 497-525.
%
\bibitem{fried0} E.\,M.\,Friedlander, Computations of $K$-theories of finite fields, Topology 15 (1976), 87-109.
%
%\bibitem{topmodels} W.~Dwyer, E.M.~Friedlander, Topological Models for Arithmetic, Topology 33 (1994), 1-24.
%
\bibitem{fried} E.\,M.\,Friedlander, Etale homotopy of simplicial schemes,
Annals of Mathematical Studies, vol. 104, Princeton University Press, 1982.
%
\bibitem{goerss} P.\,G.\,Goerss, Homotopy Fixed Points for Galois Groups, in: The \v{C}ech centennial (Boston, 1993), Contemporary Mathematics, vol. 181, 1995, 187-224.
%
\bibitem{gj} P.\,G.\,Goerss, J.\,F.\,Jardine, Simplicial Homotopy Theory, Birkh\"auser Verlag, 1999.
%
%\bibitem{sga1} A.\,Grothendieck et al., Rev\^etements \'etales et groupe fondamental (SGA 1), Lecture Notes in Mathematics, vol. 224, Springer-Verlag, 1971.
%
\bibitem{grothendieck} A.\,Grothendieck, Letter to G. Faltings (June 1983), in: P.\,Lochak, L.\,Schneps (eds.), Geometric Galois Actions; 1. Around Grothendieck's Esquisse d'un Programme, London Math. Soc. Lect. Note Ser. 242, Cambridge Univ. Press, 1997.
%
\bibitem{hs} Y.\,Harpaz, T.\,M.\,Schlank, Rational points and homotopy obstructions, in: A.\,N.\,Skorobogatov (ed.), Torsors, \'Etale Homotopy and Applications to Rational Points, London Math. Soc. Lect. Note Ser. 405, Cambridge University Press, 2013, pp. 280-394.
%
%\bibitem{modelpross} D.\,C.\,Isaksen, A model structure on the category of pro-simplicial sets, Trans. Amer. Math. Soc. 353 (2001), no. 7, 2805-2841.
%
%\bibitem{kim} M.\,Kim, Galois theory and Diophantine geometry, preprint, 2009, arXiv:0908.0533v1.
%
\bibitem{lannes} J.\,Lannes, Sur les espaces fonctionnels dont la source est le classifiant d'un p-groupe abŽlien ŽlŽmentaire, With an appendix by Michel Zisman, IHES Publ. Math. 75 (1992), 135-244.
%
\bibitem{miller} H.\,Miller, The Sullivan conjecture on maps from classifying spaces, Ann. of Math. 120 (1984), 39-87.
%
%\bibitem{mochizuki1} S.\,Mochizuki, The local pro-$p$ anabelian geometry of curves. Invent. Math. 138 (1999), 319-423.
%
%\bibitem{mochizuki2} S.\,Mochizuki, Topics surrounding the anabelian geometry of hyperbolic curves, Galois groups and fundamental groups, 119-165, Math. Sci. Res. Inst. Publ., 41, Cambridge Univ. Press, 2003.
%
\bibitem{ensprofin} F.\,Morel, Ensembles profinis simpliciaux et interpr\'etation g\'eom\'etrique du foncteur T, Bull. Soc. Math. France 124 (1996), 347-373.
%
\bibitem{pal1} A.\,Pal, The real section conjecture and Smith's fixed point theorem for pro-spaces, Journal of the London Mathematical Society 83 (2011), 353-367.
%
\bibitem{pal2} A.\,Pal, Rational points, $R$-equivalence and \'etale homotopy of algebraic varieties, preprint, arXiv:1002.1731.
%
%\bibitem{profinhom} G.\,Quick, Profinite homotopy theory, Doc. Math. 13 (2008), 585-612.
%
\bibitem{gspaces} G.\,Quick, Continuous group actions on profinite spaces, J. Pure Appl. Algebra 215 (2011), 1024-1039. 
%
\bibitem{completion} G.\,Quick, Some remarks on profinite completion of spaces, in: H.\,Nakamura, F.\,Pop, L.\,Schneps, A.\,Tamagawa (eds.), Galois-Teichm\"uller Theory and Arithmetic Geometry, Advanced Studies in Pure Mathematics, vol. 63, Mathematical Society of Japan, 2012, pp. 413-448.
%
\bibitem{gspectra} G.\,Quick, Profinite $G$-spectra, Homology, Homotopy and Applications 15 (2013), 151-189. 
%
\bibitem{homfixedlt} G.\,Quick, Continuous homotopy fixed points for Lubin-Tate spectra, Homology, Homotopy and Applications 15 (2013), 191-222.
%
%\bibitem{homalg} D.\,G.\,Quillen, Homotopical algebra, Lecture Notes in Mathematics, vol. 43, Springer 1967.
%
%\bibitem{ribes} L.\,Ribes, P.\,Zalesskii, Profinite Groups, Ergebnisse der Mathematik und ihrer Grenzgebiete, vol.~40, Springer Verlag, 2000.
%
%\bibitem{stix} J.\,Stix, The Brauer-Manin obstruction for sections of the fundamental group, preprint, 2009, arXiv:0910.5009v1.
%
\bibitem{stixbook} J.\,Stix, Rational Points and Arithmetic of Fundamental Groups
Evidence for the Section Conjecture, Springer Lecture Notes in Mathematics 2054, xx+pp. 249, Springer, 2013.
%
\bibitem{sullivan} D.\,Sullivan, Genetics of Homotopy Theory and the Adams Conjecture, Ann. of Math. 100 (1974), 1-79.
%
\bibitem{holimlim} R.\,W.\,Thomason, The homotopy limit problem, Proceedings of the Northwestern Homotopy Theory Conference (Evanston, Ill., 1982), 407-419, Contemp. Math., 19, Amer. Math. Soc., Providence, R.I., 1983.
%
%\bibitem{kirsten1} K.\,Wickelgren, Lower Central Series Obstructions to Homotopy Sections of Curves over Number Fields, preprint, 2009.
%
\bibitem{kirsten2} K.\,Wickelgren, $2$-Nilpotent Real Section Conjecture, to appear in Math. Ann.
%
\end{thebibliography}
\end{document}